\documentclass[11pt,a4paper]{amsart}
\usepackage{mathrsfs}
\usepackage{amsfonts}
\usepackage[active]{srcltx}
\usepackage{enumerate}
\usepackage[parfill]{parskip}

\usepackage{amsmath,amssymb,xspace,amsthm}

\usepackage[utf8]{inputenc}

\usepackage{amssymb, graphicx}
\input xy
\xyoption{arrow} \xyoption{matrix}
\usepackage{amsfonts}
\usepackage{amsmath}
\usepackage{latexsym}
\usepackage{amscd}
\usepackage{verbatim}
\newcommand{\bc}{\mathbb{C}}
\newcommand{\bz}{\mathbb{Z}}

\newcommand{\p}{\frac{\partial}}
\newcommand{\x}{\partial x}
\newtheorem{theorem}{Theorem}[section]
\newtheorem{proposition}[theorem]{Proposition}
\newtheorem{lemma}[theorem]{Lemma}

\theoremstyle{remark}

\numberwithin{equation}{section}
\def\mg{\mathfrak{g}}
\def\mh{\mathfrak{h}}

\def\sl{\mathfrak{sl}}

\def\sl{\mathfrak{sl}}
\def\md{\mathcal{D}}
\def\C{\mathbb{C}}
\def\Z{\mathbb{Z}}
\def\N{\mathbb{N}}
\def\bt{{\bf{t}}}
\def\bm{{\bf{m}}}
\def\bs{{\bf{s}}}
\def\bx{{\bf{x}}}
\def\bn{{\bf{n}}}
\def\br{{\bf{r}}}

\begin{document}
\title[Localization of free  field realizations]{Localization of highest weight modules of  a class of  Extended Affine Lie Algebras
}
\author{Genqiang Liu,
Yang Li, Yihan Wang}
\date{}\maketitle

\begin{abstract}In 2006, Gao and Zeng \cite{GZ} gave the free  field realizations
of highest weight modules over  a class of  extended affine Lie algebras. In the present paper,
applying  the technique of localization to those free  field realizations, we construct  a class of new  weight modules over the extended affine Lie algebras.
We give necessary and sufficient conditions for these modules to be irreducible. In this way, we construct  free field realizations for a class of simple weight modules with infinite weight multiplicities over the extended affine Lie algebras.
\end{abstract}

\vskip 5pt
\noindent
{\em 2010  Math. Subj. Class.:}
 17B10, 17B20,
17B65, 17B66, 17B68

\vskip 10pt

\allowdisplaybreaks

\section{Introduction}

In recent years, extended affine Lie algebras (EALAs) have been
studied in great detail.  EALAs are  Lie algebras which have a
non-degenerate invariant form, a self-centralizing finite
dimensional ad-diagonalizable abelian subalgebra (i.e., a Cartan
subalgebra), a discrete irreducible root system, and ad-nilpotency
of non-isotropic root spaces (see \cite{AABGP,BGK} for
definitions and structure theory).
There are many EALAs which allow not only the Laurent polynomial
algebras as co-ordinate algebras but also quantum tori, Jordan tori
and Octonion tori as co-ordinate algebras depending on the type of
algebras (see
\cite{AABGP,BGK,Y1,Y2,Y3}). For
instances, EALAs of type $A_{d-1}$ are tied up with the Lie algebra
$\mathfrak {gl}_d(\C_q)$. Quantum tori are important algebras in the theories of algebra and non-commutative geometry, see \cite{MP} and \cite{Ma}.
 To get an
EALA one has to form appropriate central extension of $\mathfrak
{gl}_d(\C_q)$ and add certain outer derivations (just like one
obtains an affine Kac-Moody Lie algebra from a  loop algebra  by
forming a one-dimensional central extension and then adding the
degree derivation).
Unlike the affine Lie algebras, the
representation theory of EALAs is far from well developed.
See \cite{DVF,E,ER,EZ,G,G1,G2,G3,L} for several interesting results on the representation theory for  the extended affine Lie algebras.

Free field realizations of Lie algebras play important role in representation
theory and conformal field theory. For an affine Lie algebra $\mathfrak{L}$, imaginary Verma modules of $\mathfrak{L}$ arise from non-standard partitions of the root system of $\mathfrak{L}$, see \cite{F}.
A $q$-version of imaginary Verma modules for the quantum groups of type $U_q(A^{(1)}_1)$ was constructed in \cite{CFKM}, and was further studied in \cite{CFM1,CFM2}.
Free field realizations of imaginary Verma modules over the affine Lie
algebra $A^{(1)}_1$ were constructed  in \cite{JK} for zero central charge and in \cite{W} for arbitrary
central charge. In \cite{FGR}, the localization technique was used to construct a new family of free field realizations.
 In particular, the (twisted) localization of imaginary Verma modules, for the Kac-Moody Lie algebra $A^{(1)}_1$, provides new irreducible weight dense modules with infinite weight multiplicities.
 In \cite{GZ}, the free  field realizations
of highest weight modules over  the  extended affine Lie algebra $\widehat{\frak{gl}_{2}(\bc_q)}$ were first given. In \cite{Z}, those  realizations were
generalized to $\widehat{\frak{gl}_{d}(\bc_q)}$ and the simplicity of the corresponding  modules was also given.
In the present paper, inspired by the idea of \cite{FGR},  we will construct  free field realizations for a class of simple weight modules with infinite weight multiplicities over the extended affine Lie algebras.

The organization of the paper  is as follows. In Section 2, we recall the  free  field realizations
of highest weight modules and their simplicity. In Subsection 3.1, we collect some preliminary results on the twisted localization.
In Subsection 3.2, we decide the simplicity of $\md_\bm M/M$. We show that $\md_\bm M/M$ is simple if and only if $\mu \notin \Z$, see Theorem \ref{Maintheorem1}. In Subsection 3.3, we give the simplicity of $\md_\bm^bM$ in the case $b\not\in\Z$.
We prove that $\md_\bm^bM$ is irreducible if and only if $b+\mu \notin \Z$, see Theorem \ref{maintheorem2}.

We denote by $\mathbb{Z}$, $\mathbb{Z}_+$, $\mathbb{N}$ and
$\mathbb{C}$ the sets of  all integers, nonnegative integers,
positive integers and complex numbers, respectively. For any Lie
algebra ${L}$, we denote  its
universal enveloping algebra by $U(L)$.
\section{Preliminaries}

 Let $q$ be a non-zero complex number. Let $\C_q$ be the
associative algebra generated by $t_1^{\pm1}, t_2^{\pm1}$ subject to the relation
$$t_1t_1^{-1}=t_1^{-1}t_1=t_2t_2^{-1}=t_2^{-1}t_2=1, t_2t_1=qt_1t_2.$$

For any $\bm=(m_1,m_2)\in\Z^2$, denote ${\bt}^{\bm}=t_1^{m_1}t_2^{m_2}$.
Then it is clear that $${\bt}^{\bm}{\bt}^{\bn}=q^{m_2n_1}\bt^{\bm+\bn}=q^{m_2n_1-m_1n_2}t^{\bn}t^{\bm}$$ for any $\bm,\bn\in\Z^2$.

Let $\frak{gl}_2(\C_q):=\frak{gl}_2(\C)\otimes_\C \C_q$ be the general linear Lie algebra coordinated by $\C_q$. Denote $X(\bm)=X\otimes t^\bm$ for
$X\in\frak{gl}_2(\C), \bm\in \Z^2$. We identify $X$ with $X(0)$.
Then the  Lie bracket of $\frak{gl}_2(\C_q)$ is given by
\begin{align*} & [e_{ij}(\bm), e_{kl}(\bn)]
=\delta_{jk}q^{m_2n_1}e_{il}(\bm+\bn)-\delta_{il}q^{n_2m_1}e_{kj}(\bm+\bn)
\end{align*}
for $\bm,\bn\in\bz^2$, $1\leq i, j, k, n\leq 2$, where
$e_{ij}$ is the matrix whose $(i, j)$-entry is $1$ and $0$
elsewhere.

Clearly $\frak{gl}_2(\C_q)$  is $\Z^2$-graded and to reflect this fact we add degree derivations.
Let $\widehat{\frak{gl}_{2}(\bc_q)} ={\frak{gl}_{2}(\bc_q)} \oplus \bc d_1
\oplus \bc d_2$ and extend the Lie bracket as
$$[d_1, X(\bm)]=m_1X(\bm),\ \  [d_2, X(\bm)]=m_2X(\bm).$$

The Lie subalgebra $[{\frak{gl}_{2}(\bc_q)},{\frak{gl}_{2}(\bc_q)}]  \oplus \bc d_1
\oplus \bc d_2$ of $\mg$ is called an extended
affine Lie algebra of type $A_{1}$ with nullity $2$. ( See [AABGP] and [BGK]
for definitions).

Let $\frak{h}=\C e_{11}\oplus \C e_{22}\oplus \bc d_1\oplus \bc d_2$ which is a Cartan subalgebra  of
$\widehat{\frak{gl}_{2}(\bc_q)}$. A $\mg$-module  $M$ is called a weight module if $\mh$ acts diagonally on  $M$, i.e.
$ M=\oplus_{\lambda\in \mh^*} M_\lambda,$
where $M_\lambda:=\{v\in M \mid xv=\lambda(x) v, \forall\ x\in \mh\}.$  A nonzero element $v\in M_\lambda$ is called a weight vector.

Let us denote
$$\frak{n}_+=e_{12}\otimes \bc_q,\ \  \frak{n}_-=e_{21}\otimes \bc_q,$$ $$\mathcal{H}=(e_{11}\otimes \C_q) \oplus (e_{22}\otimes \C_q)\oplus  \bc d_1\oplus \bc d_2.$$
Then $$\mg=\frak{n}_+\oplus \mathcal{H}\oplus \frak{n}_-.$$
For a weight $\mg$-module $M$, a weight vector $v$ in $M$ is called a highest weight vector if $\frak{n}_+v=0$.
The module $M$ is called a highest weight module if it is generated by  a highest weight vector.

Next, we will recall a class of  highest weight modules over
$\widehat{\frak{gl}_{2}(\bc_q)}$ defined by free fields, see \cite{GZ}.

 Let
 $$\bc [\mathbf{x} ]=\bc [x_{\bm}, \bm\in\bz^2]$$ be a polynomial ring with infinitely many
variables $x_{\bm}$.  Denote by $\mathcal{A}(\bx)$ the associative algebra of formal power series of differential operators on
$\bc [\mathbf{x} ]$.  By \cite{GZ}, there is an algebra homomorphism $$\Phi: U(\mg)\rightarrow \text{End}(\bc [\mathbf{x} ])$$ defined by:
\begin{eqnarray*}
e_{21}(\bn)&\mapsto&x_\bn ,\\
e_{12}(\bn)&\mapsto&q^{-n_1n_2}\mu\p{\x_{-\bn}}
-\sum_{\bs,\br\in \bz^2}
q^{n_2r_1+s_2n_1+s_2r_1}x_{\bn+\bs+\br}\p{\x_\bs}\p{\x_{\br}},\\
e_{22}(\bn)&\mapsto&\sum_{\bs\in \bz^2} q^{s_1n_2}x_{\bn+\bs}\p{\x_\bs},\\
e_{11}(\bn)&\mapsto&\mu\delta_{\bn,0}-\sum_{\bs\in \bz^2} q^{s_2n_1}x_{\bn+\bs}\p{\x_\bs},\\
 D_1&\mapsto&\sum_{\bs\in \bz^2}s_1x_\bs\p{\x_\bs},\\
 D_2&\mapsto&\sum_{\bs\in \bz^2}s_2x_\bs\p{x_\bs},
\end{eqnarray*}
where $\mu\in\C$.

Via the homomorphism $\Phi$, $\C [\mathbf{x} ]$ can be viewed as  a module over $\mg$.  We can see that $\bc [\mathbf{x} ]$ is a module of $\mg$ generated by $1$, and $e_{12}(\bn)1 = 0$ for any $\bn\in\Z^2$.
Hence $\bc [\mathbf{x} ]$ is a highest weight module of $\mg$, and $1$ is a highest weight
vector such that $e_{11}1=\mu, e_{22}1=0$.
The following result was given by Zeng, see Corollary 4.1 in \cite{Z}.
\begin{theorem}The $\mg$-module $\bc [\mathbf{x} ]$ is irreducible if and only if $\mu\neq 0$.
\end{theorem}

\section{The simplicity of the twisted localization}

In this section, we will apply the twisted localization functor to the $\mg$-module $\bc [\mathbf{x} ]$ to obtain new irreducible modules with infinite dimensional weight spaces.
The twisted localization functor was introduced by O. Mathieu to classify irreducible cuspidal modules over finite dimensional simple Lie algebras, see [M].
\subsection{The definition of the twisted localization}

Let $U(\mg)$ be the universal enveloping algebra of $\mg$. For a fixed $\bm\in \Z^2$, since $e_{21}(\bm)$ is an ad nilpotent element of $U(\mg)$, $F=\{e_{21}(\bm)^i\mid i\in \Z^+\}$ is a left and
right Ore subset of $U(\mg)$. We denote  the Ore localization of $U(\mg)$ with respect to $F$ by $U^F$.
For a $\mg$-module $M$, we denote  by $\md_\bm M$ the $F$-localization, i.e., $\md_\bm M= U^F\otimes_U M$.

 For $b\in\C $ and $u\in U^F$, we set
$$\Theta_b(u)=\sum_{j\geq 0}\binom{b}{j}(\text{ad}e_{21}(\bm))^j (u) e_{21}(\bm)^{-j},$$
where
$\binom{b}{j}= \frac{b(b-1)\cdots(b-j+1)}{j!}$ .
 Since $\text{ad}e_{21}(\bm)$ is locally nilpotent on $U^F$, the sum above
is actually finite. It is known that $\Theta_b$ is an automorphism of $U^F$. Note that for $b=k\in \Z$, we have $\Theta_b(u)=e_{21}(\bm)^kue_{21}(\bm)^{-k}$.

\begin{proposition}
For $b\in\C^*$, we have that
\begin{align*}
\Theta_b(e_{21}(\bn))=&\ e_{21}(\bn),\\
\Theta_b(e_{12}(\bn))=&\ e_{12}(\bn)+b\big(q^{m_2n_1}e_{22}(\bm+\bn)-q^{m_1n_2}e_{11}(\bm+\bn)\big)e_{21}(\bm)^{-1}\\
& -b(b-1)q^{m_2n_1+m_2m_1+m_1n_2}e_{21}(2\bm+\bn)e_{21}(\bm)^{-2},\\
\Theta_b(e_{11}(\bn))=&\ e_{11}(\bn)+bq^{m_2n_1}e_{21}(\bm+\bn)e_{21}(\bm)^{-1},\\
\Theta_b(e_{22}(\bn))=&\ e_{22}(\bn)-bq^{m_1n_2}e_{21}(\bm+\bn)e_{21}(\bm)^{-1},\\
\Theta_b(D_1)=& D_1-bm_1,\\
\Theta_b(D_1)=& D_2-bm_2,
\end{align*}
for any $\bn\in\Z^2$.
\end{proposition}
\begin{proof} First, from  $[e_{21}(\bm),e_{21}(\bn)]=0,$
we have that $$\Theta_b(e_{21}(\bn))=e_{21}(\bn).$$
Next, from $(\text{ad}e_{21}(\bm))^3(e_{12}(\bn))=0$, we obtain that
\begin{align*}
\Theta_b(e_{12}(\bn))=&\ e_{12}(\bn)+b[e_{21}(\bm),e_{12}(\bn)]e_{21}(\bm)^{-1}\\
&+\frac{1}{2}b(b-1)[e_{21}(\bm),[e_{21}(\bm),e_{12}(\bn)]]e_{21}(\bm)^{-2}\\
=&\ e_{12}(\bn)+b\big(q^{m_2n_1}e_{22}(\bm+\bn)-q^{m_1n_2}e_{11}(\bm+\bn)\big)e_{21}(\bm)^{-1}\\
&+\frac{1}{2}b(b-1)[e_{21}(\bm),[e_{21}(\bm),e_{12}(\bn)]]e_{21}(\bm)^{-2}\\
=&\ e_{12}(\bn)+b\big(q^{m_2n_1}e_{22}(\bm+\bn)-q^{m_1n_2}e_{11}(\bm+\bn)\big)e_{21}(\bm)^{-1}\\
&-b(b-1)q^{m_2n_1+m_2m_1+m_1n_2}e_{21}(2\bm+\bn)e_{21}(\bm)^{-2}.
\end{align*}
Finally, from $(\text{ad}e_{21}(\bm))^2(e_{11}(\bn))=(\text{ad}e_{21}(\bm))^2(e_{22}(\bn))=0$, we get that
\begin{align*}
\Theta_b(e_{11}(\bn))=e_{11}(\bn)+bq^{m_2n_1}e_{21}(\bm+\bn)e_{21}(\bm)^{-1},
\end{align*}
\begin{align*}
\Theta_b(e_{22}(\bn))=e_{22}(\bn)-bq^{m_1n_2}e_{21}(\bm+\bn)e_{21}(\bm)^{-1}.
\end{align*}
\end{proof}

For a $U^F$-module $M$, by $\Phi^b_\bm M$ we denote the $U^F$-module $M$ twisted by the action
$$ u\cdot v=\Theta_b(u)v, $$ where $v\in M, u\in U^F$.

In what follows, for the $\mg$-module $M=\C[\bx]$, we define $\md_\bm^bM=\Phi^b_\bm \md_\bm M$. Restricted to $U(\mg)$, $\md_\bm^bM$
 becomes a $\mg$-module, which is still denoted by $\md_\bm^bM$.

For $\bm\in \Z^2$, set $\mathcal{E}_\bm=\sum\limits_{\bs,\br\in \bz^2}
q^{-m_2r_1-s_2m_1+s_2r_1}x_{-\bm+\bs+\br}\p{\x_\bs}\p{\x_{\br}}$. Then
$$\Phi(e_{12}(-\bm))=q^{-m_1m_2}\mu\p{\x_{\bm}}-\mathcal{E}_\bm.$$
The  following lemma is a standard fact following from that $e_{12}(\bn)$ is an ad-locally nilpotent element in $U(\mg)$.
\begin{lemma} \label{Nil}If $M$ is an irreducible $\mg$-module, then the action of $e_{12}(\bn)$ on $M$ is torsion free or locally  nilpotent, for any $\bn\in\Z^2$.
\end{lemma}

In the module $\md_\bm M$, we have the following useful formulas.

\begin{lemma}\label{3.5} Let  $v_0:=x_{\bn_1}^{i_1}\cdots x_{\bn_k}^{i_k}$, where $\bn_1,\cdots,\bn_k\neq \bm$, $i_1,\cdots,i_k\in\N$. Then
\begin{itemize}
\item[(1).]$e_{12}(-\bm)x_\bm^{-i}v_0=-x_m^{-i} \mathcal{E}_\bm(v_0)+(\sum\limits_{j=1}^k2i_j-\mu-i-1)iq^{-m_2m_1}x_\bm^{-i-1}v_0$,
   \item[(2).]$e_{12}(-\bm)x_\bm^{-i}=(-\mu-i-1)iq^{-m_2m_1}x_\bm^{-i-1}$,
   \item[(3).]$\mathcal{E}_\bm^{d+1} v_0=0$, where $d=i_1+\cdots+i_k$.
\end{itemize}
\end{lemma}
\begin{proof}For $\bn\in\Z^2$, we can compute that
\begin{align*}& e_{12}(-\bn)x_\bm^{-i}x_{\bn_1}^{i_1}\cdots x_{\bn_k}^{i_k}\\
=&(q^{-n_1n_2}\mu\p{\x_{\bn}}
-\sum_{\bs,\br\in \bz^2}
q^{-n_2r_1-s_2n_1+s_2r_1}x_{-\bn+\bs+\br}\p{\x_\bs}\p{\x_{\br}})x_\bm^{-i}x_{\bn_1}^{i_1}\cdots x_{\bn_k}^{i_k}\\
=&-i\delta_{\bn,\bm}q^{-n_1n_2}\mu x_\bm^{-i-1}x_{\bn_1}^{i_1}\cdots x_{\bn_k}^{i_k}+x_\bm^{-i}q^{-n_1n_2}\mu\p{\x_{\bn}}(x_{\bn_1}^{i_1}\cdots  x_{\bn_k}^{i_k})\\
&+\sum_{\bs\in \bz^2}iq^{-n_2m_1-s_2n_1+s_2m_1}x_{-\bn+\bs+\bm}\p{\x_\bs}x_\bm^{-i-1}x_{\bn_1}^{i_1}\cdots x_{\bn_k}^{i_k}\\
& -\sum_{j=1}^k\sum_{\bs\in \bz^2}q^{-n_2n_{j1}-s_2n_1+s_2n_{j1}}i_jx_{-\bn+\bs+\bn_j}\p{\x_\bs}x_\bm^{-i}x_{\bn_1}^{i_1}\cdots x_{\bn_j}^{i_j-1}\cdots x_{\bn_k}^{i_k}\\
=&-x_\bm^{-i-1}i\delta_{\bn,\bm}q^{-n_1n_2}\mu x_{\bn_1}^{i_1}\cdots x_{\bn_k}^{i_k}\\
&-x_\bm^{-i-2}i(i+1)q^{-n_2m_1-m_2n_1+m_2m_1}x_{-\bn+2\bm}x_{\bn_1}^{i_1}\cdots x_{\bn_k}^{i_k}\\
& +x_\bm^{-i-1}\sum_{j=1}^ki_jiq^{-n_2m_1-n_{j2}n_1+n_{j2}m_1}x_{-\bn+\bn_j+\bm}x_{\bn_1}^{i_1}\cdots x_{\bn_j}^{i_j-1}\cdots x_{\bn_k}^{i_k}\\
&+x_\bm^{-i-1}\sum_{j=1}^kii_jq^{-n_2n_{j1}-m_2n_1+m_2n_{j1}}x_{-\bn+\bm+\bn_j}x_{\bn_1}^{i_1}\cdots x_{\bn_j}^{i_j-1}\cdots x_{\bn_k}^{i_k}\\
& +x_\bm^{-i}e_{12}(-\bn)(x_{\bn_1}^{i_1}\cdots x_{\bn_k}^{i_k}),
\end{align*}
Take $\bn=\bm$ in the above computation, (1) follows. (2), (3) are clear.
\end{proof}

\subsection{The simplicity for the case $b\in\Z$}

Let $\C[x_\bm^{-1}, \bx, \hat x_\bm]$ be the polynomial ring in the variables $x_{\bm}^{-1}, x_{\bn}, \bn\neq \bm$.
\begin{lemma} \label{Generator}Let $\mu\notin \Z$. Then
 the $\mg$-module $\md_\bm M/M$  is generated by $x_\bm^{-1}$.
\end{lemma}
\begin{proof} As vector spaces $\md_\bm M/M\cong x_\bm^{-1}\C[x_\bm^{-1}, \bx, \hat x_\bm]$. For any $i\in\N$, we have that
\begin{align*}e_{12}(-\bm)^ix_\bm^{-1}=& q^{-im_1m_2}(\mu\p{\x_\bm}-x_\bm\p{\x_\bm}\p{\x_\bm})^ix_\bm^{-1}\\
=& (-1)^iq^{-im_1m_2}i!(\mu+2)\cdots (\mu+1+i)x_\bm^{-1-i}.
\end{align*}
The condition that $\mu\notin \Z$ forces that $x_\bm^{-1-i}$ can be generated by $x_\bm^{-1}$.
Then from  $ e_{21}(\bn)x_\bm^{-1-i}=x_\bm^{-1-i}x_\bn$ for any $\bn\in\Z^2$ with $\bn\neq \bm$, we see that $x_\bm^{-1}$ can generate  $\md_\bm M/M$.
\end{proof}
\begin{theorem}\label{Maintheorem1}Let $\bm\in \Z^2, b\in \C$.
 If $b\in \Z$, then $\md_\bm^bM\cong \md_\bm M$ and $\md_\bm M/M$ is irreducible if and only if $\mu \notin \Z$.
\end{theorem}
\begin{proof}In the case $b=k\in \Z$, define the following linear map
$$ \rho: \md_\bm M\rightarrow \md_\bm^bM; v\mapsto e_{21}(\bm)^k v ,$$ which is clearly a bijection.
From $$\rho(uv)=e_{21}(\bm)^k uv= e_{21}(\bm)^ku e_{21}(\bm)^{-k}e_{21}(\bm)^k v = u\cdot \rho(v),$$ where $u\in U^F, v\in \md_\bm M$,
we see that $ \rho$ is a $\mg$-module isomorphism.

($\Rightarrow$) Let $\mu\in\Z$.

{\bf Claim 1:} There exists a  nonzero  $w\in \md_\bm M/M$ such that $e_{12}(-\bm)\cdot w=0$.

Choose a positive integer  $d$ such that $-d+\mu+1< 0$.  Choose $\bn\in\Z^2$ such that $n_1m_2-n_2m_1\neq 0$. Let  $l=2d-\mu-1$ and
 $a_j=q^{-m_1m_2}(\mu+j+l+1-2d)(j-l)$ which is nonzero by the choice of $d$, for $j: 1\leq j\leq d$.
Set $w_j=\frac{1}{a_1\cdots a_j}\mathcal{E}_\bm^{j} (x_{\bn}^d)$ for $j: 0\leq j\leq d+1$. Note that  $w_0=x_{\bn}^d, w_{d+1}=0$, $\mathcal{E}_\bm w_j=a_{j+1}w_{j+1}$,
and $\sum\limits_{\br\in \bz^2}x_\br\p{\x_\br}w_j=(d-j)w_j$, for any $0\leq j\leq d-1$. By the choice of $\bn$, we see that $\p{\x_{\bm}}(w_j)=0$ for $j: 0\leq j\leq d$.

Let $w=\sum_{j=0}^d x^{j-l}_\bm w_j$. We will show that $w$ is annihilated by
$ e_{12}(-\bm)$.

Firstly, we have
\begin{align*}\mathcal{E}_\bm x^{j-l}_\bm w_j=&(\mathcal{E}_\bm x^{j-l}_m )w_j+  x^{j-l}_\bm (\mathcal{E}_\bm w_j)\\
&+\sum_{\bs,\br\in \bz^2}
q^{-m_2r_1-s_2m_1+s_2r_1}x_{-\bm+\bs+\br}(\p{\x_\bs}x^{j-l}_m)(\p{\x_{\br}}w_j)\\
&+\sum_{\bs,\br\in \bz^2}
q^{-m_2r_1-s_2m_1+s_2r_1}x_{-\bm+\bs+\br}(\p{\x_\br}x^{j-l}_\bm)(\p{\x_{\bs}}w_j)\\
=& q^{-m_1m_2}(j-l)(j-l-1)x^{j-l-1}_\bm w_j+a_{j+1}x^{j-l}_\bm w_{j+1}\\
& +2\sum_{\br\in \bz^2}q^{-m_2m_1}(j-l)x_\br x^{j-l-1}_\bm(\p{\x_{\br}} w_j)\\
=& q^{-m_1m_2}(j-l)(j-l-1)x^{j-l-1}_\bm w_j+a_{j+1}x^{j-l}_\bm w_{j+1}\\
& +2(d-j)q^{-m_2m_1}(j-l)x^{j-l-1}_\bm w_j\\
=& q^{-m_1m_2}(j-l)(-j-l-1+2d)x^{j-l-1}_\bm w_j+a_{j+1}x^{j-l}_\bm w_{j+1}.
\end{align*}

Consequently,
\begin{align*}& e_{12}(-\bm)w =\sum_{j=0}^d(q^{-m_1m_2}\mu\p{\x_{\bm}}-\mathcal{E}_\bm) x^{j-l}_\bm w_j\\
=& \sum_{j=0}^d q^{-m_1m_2}\mu\big((j-l)x^{j-l-1}_\bm w_j+  x^{j-l}_\bm( \p{\x_{\bm}} w_j)\big)\\
& -\sum_{j=0}^d   q^{-m_1m_2}(j-l)(-j-l-1+2d)x^{j-l-1}_\bm w_j-\sum_{j=0}^d a_{j+1}x^{j-l}_\bm w_{j+1}\\
=& \sum_{j=0}^d q^{-m_1m_2}(\mu+j+l+1-2d)(j-l)x^{j-l-1}_\bm w_j -\sum_{j=1}^d  a_{j}x^{j-l-1}_\bm w_{j}\\
=& q^{-m_1m_2}(\mu+l+1-2d)(-l)x^{-l-1}_\bm w_0=0,
\end{align*}
in the third equality, we have used the fact that $ \p{\x_{\bm}}(w_j)=0$.

If $\md_\bm M/M$ is an irreducible $\mg$-module, then  by Lemma \ref{Nil}, $ e_{12}(-\bm)$ acts locally nilpotently on $\md_\bm M/M$. On the other hand,
for a enough large   integer $j$,  by (3) in Lemma \ref{3.5}, we have
$$e_{12}(-\bm)^ix_\bm^{-j}=(-1)^iq^{-im_1m_2}j\cdots (j+i-1)(\mu+j+1)\cdots (\mu+j+i)x_\bm^{-j-i}$$ which is a nonzero element in
$\md_\bm M/M$, for any $i\in\N$. This is a contradiction. So  $\md_\bm M/M$ is reducible when $\mu\in \Z$.

($\Leftarrow$): Let $v$ be any nonzero element in $\md_\bm M/M$.  By Lemma \ref{Generator}, it suffices to show that there exists $u\in U(\mg)$ such
that $uv=x_\bm^{-1}$.

After applying some power of  $e_{21}(\bm)$ to $v$ if necessary, we may assume that $v=x_{\bm}^{-1}f$, where $f$ is a homogeneous polynomial in $\C[\bx, \hat x_\bm]$ of degree $d$.

\noindent{\bf Claim 2:} For such $v$, there exists $u\in U(\mg)$ such that $uv=x_{\bm}^{-1}f_1\neq 0$, where $f_1$ is a homogeneous polynomial in $\C[\bx, \hat x_\bm]$ whose  degree is smaller than $d$.

First, we consider the case that $v$ is a monomial, i.e., $v=x_\bm^{-1}x_{\bn_1}^{i_1}\cdots x_{\bn_k}^{i_k}$. From the proof of Lemma \ref{3.5},
we have that
\begin{align*}& e_{12}(-\bn_1)x_\bm^{-1}x_{\bn_1}^{i_1}\cdots x_{\bn_k}^{i_k}\\
=&-x_\bm^{-3}2q^{-n_{12}m_1-m_2n_{11}+m_2m_1}x_{-\bn_1+2\bm}x_{\bn_1}^{i_1}\cdots x_{\bn_k}^{i_k}\\
& +x_\bm^{-2}\sum_{j=1}^ki_jq^{-n_{12}m_1-n_{j2}n_{11}+n_{j2}m_1}x_{-\bn_1+\bn_j+\bm}x_{\bn_1}^{i_1}\cdots x_{\bn_j}^{i_j-1}\cdots x_{\bn_k}^{i_k}\\
& +x_\bm^{-2}\sum_{j=1}^ki_jq^{-n_{12}n_{j1}-m_2n_{11}+m_2n_{j1}}x_{-\bn_1+n_j+\bm}x_{\bn_1}^{i_1}\cdots x_{\bn_j}^{i_j-1}\cdots x_{\bn_k}^{i_k}\\
& +x_\bm^{-1}e_{12}(-\bn_1)(x_{\bn_1}^{i_1}\cdots x_{\bn_k}^{i_k})\\
 =& x_\bm^{-3}a_{-3}+x_\bm^{-2}a_{-2}+x_\bm^{-1}a_{-1},
\end{align*}
where
\begin{align*}a_{-1}=&2i_1q^{-n_{11}n_{12}}x_{\bn_1}^{-1}x_{\bn_1}^{i_1}\cdots x_{\bn_k}^{i_k}+ e_{12}(-\bn_1)(x_{\bn_1}^{i_1}\cdots x_{\bn_k}^{i_k})\\
=& i_1(3+\mu-i_1-2i_2\cdots-2i_k)q^{-n_{11}n_{12}}x_{\bn_1}^{-1}x_{\bn_1}^{i_1}\cdots x_{\bn_k}^{i_k}\\
& + \sum_{j=2}^k\sum_{l=2}^k i_l (i_j-\delta_{j,l})q^{-n_{12}n_{j1}-n_{l2}n_{11}+n_{l2}n_{j1}}x_{-\bn_1+\bn_l+\bn_j}x_{\bn_j}^{-1}x_{\bn_l}^{-1}x_{\bn_1}^{i_1}\cdots x_{\bn_k}^{i_k},\\
a_{-2}=&\sum_{j=2}^ki_j(q^{-n_{12}m_1-n_{j2}n_{11}+n_{j2}m_1}+q^{-n_{12}n_{j1}-m_2n_{11}+m_2n_{j1}})\\
& \times x_{-\bn_1+\bn_j+\bm}x_{\bn_1}^{i_1}\cdots x_{\bn_j}^{i_j-1}\cdots x_{\bn_k}^{i_k},\\
a_{-3}=& -2q^{-n_{12}m_1-m_2n_{11}+m_2m_1}x_{-\bn_1+2\bm}x_{\bn_1}^{i_1}\cdots x_{\bn_k}^{i_k}.
\end{align*}

Let $A=-2(\sum_{j=1}^k2i_j-\mu-1)q^{-m_2m_1}$. By Lemma \ref{3.5}, we have
\begin{align*}
&(e_{12}(-\bm)e_{21}(\bm)+A)e_{12}(-\bn_1)x_\bm^{-1}x_{\bn_1}^{i_1}\cdots x_{\bn_k}^{i_k}\\
=&2(\sum_{j=1}^k2i_j-\mu-1)q^{-m_1m_2}x_\bm^{-3}a_{-3}-x_\bm^{-2}\mathcal{E}_\bm(a_{-3})\\
 & + x_\bm^{-2}(\sum_{j=1}^k2i_j-\mu-2)q^{-m_2m_1}a_{-2}-x_\bm^{-1}\mathcal{E}_\bm(a_{-2})\\
 & + Ax_\bm^{-3}a_{-3}+Ax_\bm^{-2}a_{-2}+Ax_\bm^{-1}a_{-1}\\
=&  x_\bm^{-2}b_{-2}+x_\bm^{-1}b_{-1},
\end{align*}
where
\begin{align*}
b_{-2}
= & -\mathcal{E}_\bm(a_{-3})+ (\sum_{j=1}^k2i_j-\mu-2)q^{-m_2m_1}a_{-2} +Aa_{-2},\\
b_{-1}
= &-\mathcal{E}_\bm(a_{-2})+Aa_{-1}.
\end{align*}

Let $ B=-(\sum_{j=1}^k2i_j-\mu-2)q^{-m_1m_2}$. Then  we have that
\begin{align*}
&(e_{12}(-\bm)e_{21}(\bm)+B)(e_{12}(-\bm)e_{21}(\bm)+A)e_{12}(-\bn_1)x_\bm^{-1}x_{\bn_1}^{i_1}\cdots x_{\bn_k}^{i_k}\\
=&  (e_{12}(-\bm)e_{21}(\bm)+B)(x_\bm^{-2}b_{-2}+x_\bm^{-1}b_{-1})\\
=&e_{12}(-\bm)x_\bm^{-1}b_{-2}+Bx_\bm^{-2}b_{-2}+Bx_\bm^{-1}b_{-1}\\
=& -x_\bm^{-1}\mathcal{E}_\bm(b_{-2})+(\sum_{j=1}^k2i_j-\mu-2)q^{-m_1m_2}x_\bm^{-2}b_{-2}+Bx_\bm^{-2}b_{-2}+Bx_\bm^{-1}b_{-1}\\
=& x_\bm^{-1}(Bb_{-1}-\mathcal{E}_\bm(b_{-2}))+((\sum_{j=1}^k2i_j-\mu-2)q^{-m_1m_2}+B)x_\bm^{-2}b_{-2}\\
=& x_\bm^{-1}(Bb_{-1}-\mathcal{E}_\bm(b_{-2})),
\end{align*}
where \begin{align*}Bb_{-1}-\mathcal{E}_\bm(b_{-2})=& -B\mathcal{E}_\bm(a_{-2})+BAa_{-1}+ \mathcal{E}_\bm^2(a_{-3})\\
 & -(\sum_{j=1}^k2i_j-\mu-2)q^{-m_2m_1}\mathcal{E}_\bm(a_{-2}) -A\mathcal{E}_\bm(a_{-2})\\
 =&-A\mathcal{E}_\bm(a_{-2})+BAa_{-1}+ \mathcal{E}_\bm^2(a_{-3}).
\end{align*}

We can see that the coefficients of $x_{\bn_1}^{-1}x_{\bn_1}^{i_1}\cdots x_{\bn_k}^{i_k}$ in $\mathcal{E}_\bm(a_{-2})$ and $\mathcal{E}_\bm^2(a_{-3})$ are zero.
Thus the coefficient of of $x_{\bn_1}^{-1}x_{\bn_1}^{i_1}\cdots x_{\bn_k}^{i_k}$ in $Bb_{-1}-\mathcal{E}_\bm(b_{-2})$ is
$ABi_1(3+\mu-i_1-2i_2\cdots-2i_k)q^{-n_{11}n_{12}}$ which is nonzero,
since $\mu\not\in\Z$, $A\neq 0, B\neq 0$. Hence $Bb_{-1}-\mathcal{E}_\bm(b_{-2})\neq 0$.  Therefore, the claim 2 is true in this case.

For arbitrary nonzero $v\in\md_\bm M/M$, by the action of $e_{21}(\bm)$, we assume that
$$v=x_{\bm}^{-1}\sum\limits_{j_1,\dots,j_k\in\Z_+}a_{j_1,\dots,j_k}x_{\bn_1}^{j_1}\cdots x_{\bn_k}^{j_k},$$ where  $x_{\bn_1}^{j_1}\cdots x_{\bn_k}^{j_k}\in\C[\bx, \hat x_\bm]$,
$j_1+\cdots+j_k=d$. Suppose that the maximal degree of $x_{\bn_1}$ in $v$ is $i_1$. By the above arguments, we see that
the coefficient of $x_{\bm}^{-1}x_{\bn_1}^{i_1-1}\sum\limits_{j_2+\dots + j_k =d-i_1}a_{i_1,\dots,j_k}x_{\bn_2}^{j_2}\cdots x_{\bn_k}^{j_k}$ in  $(e_{12}(-\bm)e_{21}(\bm)+B)(e_{12}(-\bm)e_{21}(\bm)+A)e_{12}(-\bn_1)v$
is $ABi_1(3+\mu+i_1-2d)q^{-n_{11}n_{12}}$ which is nonzero. So the claim 2 is true in general.

By induction on the degree of $f$ in $v=x_{\bm}^{-1}f$, we can complete the proof.

\end{proof}
\subsection{The case $b\not\in\Z$}

\begin{theorem}\label{maintheorem2}Let $m\in \Z^2, b\in \C$.
If $b\notin \Z$, then  $\md_\bm^bM$ is irreducible if and only if $b+\mu \notin \Z$.
\end{theorem}
\begin{proof}($\Rightarrow$): Let $\mu+b\in\Z$.

From \begin{align*}e_{12}(-\bm)\cdot x_\bm^{j}
=& \Big(q^{-m_1m_2}\mu\p{\x_{\bm}}
-\sum_{\bs,\br\in \bz^2}
q^{-m_2r_1-s_2m_1+s_2r_1}x_{-\bm+\bs+\br}\p{\x_\bs}\p{\x_{\br}}\\
& +2b\sum_{\bs\in \bz^2} q^{-m_1m_2}x_{\bs}\p{\x_\bs}x_\bm^{-1}-b(b+\mu-1)q^{-m_1m_2} x_\bm^{-1}\Big) x_\bm^{j}\\
=& -(j-b)(j-b-1-\mu)q^{-m_1m_2}x_\bm^{j-1},
\end{align*}
we see that $e_{12}(-\bm)x_\bm^{b+1+\mu}=0$ and $e_{12}(-\bm)x_\bm^{j}\neq 0$ for any $j<b+1+\mu$.
By Lemma \ref{Nil}, $\md_\bm^bM$ is reducible.

($\Leftarrow$): Suppose that $b+\mu \notin \Z$.  Let $V$ be nonzero submodule of $\md_\bm^bM$.

\noindent{\bf Claim :} There is an integer $p$ and some nonzero $w\in V$ such that $$\Big(e_{21}(\bm)e_{12}(-\bm)-(b-p)(2d-\mu-b-p-1)\Big)w=a x_\bm^{d}$$ for some $a\in\C$, where
$d$ is a positive integer.

Let $v\in V$ be a nonzero homogeneous polynomial in $\C[x_\bm^{-1}, \bx, \hat x_\bm]$ of degree $d$. We assume that
$v=\sum\limits_{z\leq i\leq l}x_{\bm}^i f_i$, where each $f_{i}$ is a homogeneous polynomial in $\C[\bx, \hat x_\bm]$ of degree $d-i$.
We define $d-z$ as the ${\hat x_\bm}$-degree of $v$. Note that $(e_{11}-e_{22})\cdot v=(\mu-2d+2b)v $.

From \begin{align*}&e_{21}(\bm)\cdot e_{12}(-\bm)\cdot x_\bm^{-i}x_{\bn_1}^{i_1}\cdots x_{\bn_k}^{i_k}\\
=& e_{21}(\bm)\cdot \Big(-x_\bm^{-i} \mathcal{E}_\bm(x_{\bn_1}^{i_1}\cdots x_{\bn_k}^{i_k})+(\sum\limits_{j=1}^k2i_j-\mu-i-1)iq^{-m_2m_1}x_\bm^{-i-1}x_{\bn_1}^{i_1}\cdots x_{\bn_k}^{i_k}\\
 & +2b\sum_{\bs\in \bz^2} q^{-m_1m_2}x_{\bs}\p{\x_\bs}(x_\bm^{-i-1}x_{\bn_1}^{i_1}\cdots x_{\bn_k}^{i_k})-b(b+\mu-1)q^{-m_1m_2} x_\bm^{-i-1}x_{\bn_1}^{i_1}\cdots x_{\bn_k}^{i_k}\Big)\\
=& -x_\bm^{-i+1}\mathcal{E}_\bm(x_{\bn_1}^{i_1}\cdots x_{\bn_k}^{i_k})+(i+b)(\sum\limits_{j=1}^k2i_j-b-\mu-i-1)q^{-m_1m_2} x_\bm^{-i}x_{\bn_1}^{i_1}\cdots x_{\bn_k}^{i_k}
\end{align*}
we see that $\Big(e_{21}(\bm)e_{12}(-\bm)-(b-z)(2d-\mu-b-z-1)\Big)\cdot v$ has smaller ${\hat x_\bm}$-degree than $v$. By induction  on the ${\hat x_\bm}$-degree of $v$,
we know that the above claim is true.

If $a\neq 0$, from that  $e_{12}(-\bm)\cdot x_\bm^{d}$ is a nonzero multiple of $ x_\bm^{d-1}$ and $e_{21}(\bm)\cdot x_\bm^{d}=x_\bm^{d+1}$, we have that $\C[x_\bm]\subset V$. Consequently,  $V=\md_\bm^bM$.

If $a=0$, then $e_{12}(-\bm)\cdot w= (b-p)(2d-\mu-b-p-1)e_{21}(\bm)^{-1}w$.

Let $B_i=(i+b-p)(2d-b-\mu-1-i-p)q^{-m_1m_2}$ which is nonzero for any $i\in \Z_+$.
Using  the following formula:
$$e_{21}(\bm)^{-1}e_{12}(-\bm)-e_{12}(-\bm)e_{21}(\bm)^{-1}=q^{-m_1m_2}e_{21}(\bm)^{-1}(e_{11}-e_{22})e_{21}(\bm)^{-1},$$ we obtain that
\begin{align*} &e_{12}(-\bm)\cdot e_{21}(\bm)^{-k} \cdot w\\
=& \sum_{i=0}^{k-1}e_{21}(\bm)^{-i}\cdot[e_{12}(-\bm),e_{21}(\bm)^{-1}]\cdot e_{21}(\bm)^{-(k-i-1)} \cdot w+ B_0e_{21}(\bm)^{-k-1}\cdot w\\
=& -q^{-m_1m_2}\sum_{i=0}^{k-1}e_{21}(\bm)^{-i-1}\cdot (e_{11}-e_{22})\cdot e_{21}(\bm)^{-(k-i)}\cdot w+ B_0e_{21}(\bm)^{-k-1}\cdot w\\
=&\Big(B_0- q^{-m_1m_2}\sum_{i=0}^{k-1}(\mu-2d+2b+2(k-i))\Big) e_{21}(\bm)^{-k-1} \cdot w\\
=&\Big(B_0- q^{-m_1m_2}k(\mu-2d+2b+k+1)\Big)e_{21}(\bm)^{-k-1}\cdot w\\
=& B_ke_{21}(\bm)^{-k-1}\cdot w.
\end{align*}
By induction on $k$, we have that
\begin{align*} e_{12}(-\bm)^k \cdot w= B_0B_1\cdots B_{k-1}e_{21}(\bm)^{-k}w.\end{align*}
More precisely \begin{align*} e_{12}(-\bm)^{k+1} \cdot w= &B_0B_1\cdots B_{k-1}e_{12}(-\bm)\cdot e_{21}(\bm)^{-k}\cdot w\\
=&B_0B_1\cdots B_{k-1}B_ke_{21}(\bm)^{-k-1} \cdot w.
\end{align*}

From the action of $e_{21}(\bm)$, we can assume that $w\in \C[\bx]$. By the simplicity of the $\mg$-module $\C[\bx]$, there a $u_1\in U(\mg)$ such that
$u_1 w=1$. Since $\Theta_b$ is an automorphism of $U^F$, there are  $u_2\in U(\mg)$ and $k\in \Z_+$ such that $ \Theta(u_2 e_{21}(\bm)^{-k})=u_1$.

Then $$1=u_1 w=\Theta(u_2 e_{21}(\bm)^{-k})w=u_2 \cdot e_{21}(\bm)^{-k}\cdot w=C u_2  \cdot e_{12}(-\bm)^{k}\cdot w\in V,$$
where $C=\frac{1}{B_0B_1\cdots B_{k-1}}$. Since $\md_\bm^bM$ is generated by $1$, $V=\md_\bm^bM$. Now we complete the proof.
\end{proof}

\section*{Acknowledgment}
G.L. is partially supported by NSF of China
(Grant 11301143, 11771122) and the grant at Henan University (yqpy20140044).

\

\

\noindent G.L.: School of Mathematics and Statistics, and  Institute of Contemporary Mathematics, Henan University, Kaifeng 475004, China. Email: liugenqiang@amss.ac.cn

\vspace{0.2cm}\noindent Y.L.: School of Mathematics and Statistics, Henan University, Kaifeng 475004,  China. Email: 897981524@qq.com

\vspace{0.2cm}\noindent Y.W.: School of Mathematics and Statistics, Henan University, Kaifeng 475004,  China. Email: 2276588574@qq.com


\begin{thebibliography}{99}

\bibitem[AABGP]{AABGP} B. N. Allison, S. Azam, S. Berman, Y. Gao,  A.
Pianzola, {\it Extended affine Lie algebras and their root
systems}, Memoir. Amer. Math. Soc., 126 (1997) no. 605.

\bibitem[BGK]{BGK} S. Berman, Y. Gao,  Y. Krylyuk, Quantum tori and the
structure of elliptic
quasi-simple Lie algebras, J. Funct. Anal., 135 (1996) 339--389.

\bibitem[BGT]{BGT} S. Berman,  Y. Gao,  S. Tan, A unified view of  some vertex operator constructions,
           Israel  Journal of Mathematics, 134 (2003) 29--60.

\bibitem[BL]{BL}Y. Billig, M. Lau, Irreducible modules for extended affine Lie algebras, J. Algebra, 327 (2011) 208--235.

\bibitem[BS]{BS} S. Berman and J. Szmigielski, Principal realization
for
 extended affine Lie algebra of type $\sl_2$ with coordinates in a simple
quantum torus with two variables, Contemp. Math., 248 (1999) 39--67.

\bibitem[CFM1]{CFM1} B. Cox, V.M. Futorny, K.C. Misra, Imaginary Verma modules for $U_q(\widehat{\sl(2)})$ and crystal-like bases,
J. Algebra, 481 (2017)12--35.
\bibitem[CFM2]{CFM2} B. Cox, V.M. Futorny, K.C. Misra, Imaginary Verma modules and Kashiwara algebras for $U_q(\widehat{\sl(2)})$, Contemp. Math., 506 (2010) 105--126.
\bibitem[CFKM]{CFKM} B. Cox, V.M. Futorny, S.J. Kang, D.J. Melville, Quantum deformations of imaginary Verma modules, Proc. Lond. Math. Soc. (3) 74 (1) (1997) 52-80.
\bibitem[DVF]{DVF} M. Dokuchaev, L.M. Vasconcellos Figueiredo, V. Futorny, Imaginary Verma modules for the extended affine Lie algebra $\sl_2(\C_q)$, Comm. Algebra 31 (1) (2003) 289--308.
\bibitem[E]{E}  S. Eswara Rao, A class of integrable modules for the core of EALA Coordinatized
by quantum tori, J. Algebra, 275 (2004), 59--74.


\bibitem[ER]{ER} Eswara Rao, Unitary modules for EALAs co-ordinatized by a quantum torus, Comm. Algebra, 31 (2003) 2245--2256.

\bibitem[EZ]{EZ} S. Eswara Rao, K. Zhao, Integrable representations of  Toroidal Lie algebras Co-ordinatized
by rational quantum tori, J. Algebra, 361 (2012), 225–-247.


\bibitem[F]{F}V.M. Futorny, Imaginary Verma modules for affine Lie algebras, Canad. Math. Bull., 37 (2) (1994) 213--218.
\bibitem[FF]{FF} B. Feigin, E. Frenkel, Affine Kac-Moody algebras and semi-infinite flag manifolds,
Comm. Math. Phys., 128 (1990) 161--189.

\bibitem[FGR]{FGR}V. Futorny, D. Grantcharov, R. Martins, Localization of free field realizations of affine Lie algebras,  Lett. Math. Phys., 2015, 105: 483–502.

\bibitem[G]{G} D. Grantcharov, Twisted localization of weight modules, Developments and retrospectives in Lie theory, 185–206, Dev. Math., 38, Springer, Cham, 2014.
\bibitem[G1]{G1} Y. Gao, Representations of extended affine Lie algebras
coordinatized by certain quantum tori, Compositio Mathematica, 123 (2000) 1--25.

\bibitem[G2]{G2} Y. Gao, Vertex operators arising from the homogeneous
realization for $\widehat{\frak{gl}}_{N}$, Comm. Math. Phys., 211 (2000) 745--777.

\bibitem[G3]{G3} Y. Gao, Fermionic and bosonic representations of
 the extended affine Lie algebra $\widetilde{\frak{gl}_{{}_N}(\bc_q)}$,
 Canada Math Bull., 45 (2002) 623--633.

\bibitem[GZ]{GZ}  Y. Gao and Z. Zeng, Hermitian representations
of the extended affine Lie algebra
$\widetilde{\frak{gl}_2(\Bbb{C}_q)}$, Adv. Math., 207 (2006)
244--265.



\bibitem[H-KT]{H-KT}R. H\o egh-Krohn, B. Torresani,
Classification and construction of quasi-simple Lie
algebras, J. Funct. Anal., 89 (1990) 106--136.

\bibitem[JK]{JK} H. P. Jakobsen, V. G. Kac, A new class of
unitarizable highest weight representations of infinite-dimensional
 Lie algebras, II, J. Funct. Anal., 82 (1989) 69--90.

\bibitem[L]{L} M. Lau,  Bosonic and fermionic representations
of Lie algebra central extensions,  Adv. Math., 194(2005)
225--245.

\bibitem[M]{M}  O. Mathieu, Classification of irreducible weight modules, Ann. Inst. Fourier, 50 (2000)
537--592.

\bibitem[Ma]{Ma} Y. I. Manin, Topics in non-commutative Geometry, Princeton Press, 1991.
\bibitem[MP]{MP} J. C.  Connell and J. J. Pettit, Crossed products and
multiplicative analogues of Weyl algebras, J. London Math. Soc., (2)
(1988), No.1, 47--55.

\bibitem[VV]{VV} M. Varagnolo and E. Vasserot, Double-loop algebras
and
the Fock space,  Invent. Math., 133 (1998) 133--159.

\bibitem[W]{W} M. Wakimoto,
Fock representations of the affine Lie algebra $A\sp {(1)}\sb 1$,
Comm. Math. Phys., 104 (1986) 605--609.

\bibitem[Y1]{Y1} Y. Yoshii, Co-ordinate algebras of extended affine Lie algebras of type A, J.
Algebra, 234 (2000), 128--168.
\bibitem[Y2]{Y2}  Y. Yoshii, Classification of division $\Z^n$-graded
alternative algebras, J. Algebra, 256 (2002), 25--50.
\bibitem[Y3]{Y3}  Y. Yoshii, Classification of quantum tori with involution, Canad. Math. Bull., 45 (4),
(2002), 711--731.

\bibitem[Z]{Z} Z. Zeng,
A class of irreducible modules for the extended affine Lie algebra,
Sci. China Math., 54 (2011)1089--1099.

\end{thebibliography}
\end{document}